\documentclass[a4paper,12pt]{amsart}

\newtheorem{thm}{Theorem}[section]
\newtheorem{crl}[thm]{Corollary}

\newtheorem{lmm}[thm]{Lemma}

\newtheorem{prp}[thm]{Proposition}

\theoremstyle{definition}
\newtheorem{dfn}[thm]{Definition}

\theoremstyle{remark}

\usepackage{amssymb, amsmath}
\usepackage{graphics}
\usepackage[dvips]{graphicx}
\usepackage{eepic}
\usepackage{epic}

\input{rims-mymacros.sty}

\title{A note on consistency conditions on dimer models}
\author{Akira Ishii and Kazushi Ueda}

\begin{document}
%

\maketitle


\section{Introduction}

Dimer models are introduced by string theorists
(see e.g. \cite{Franco-Hanany-Martelli-Sparks-Vegh-Wecht_GTTGBT,
Franco-Hanany-Vegh-Wecht-Kennaway_BDQGT,
Franco-Vegh_MSGTDM,
Hanany-Herzog-Vegh_BTEC,
Hanany-Kennaway_DMTD,
Hanany-Vegh})
to study supersymmetric quiver gauge theories
in four dimensions.
A dimer model is a bicolored graph on a 2-torus
which encode the information of a quiver with relations.
If a dimer model is non-degenerate,
then the moduli space $\scM_{\theta}$
of stable representations of the quiver
with dimension vector $(1, \dots, 1)$
with respect to a generic stability parameter $\theta$
in the sense of King \cite{King}
is a smooth toric Calabi-Yau 3-fold
\cite{Ishii-Ueda_08}.

Let $\scV = \bigoplus_v \scL_v$ be
the tautological bundle
on the moduli space $\scM_\theta$
and
\begin{equation} \label{eq:univ_morph}
 \phi : \bC \Gamma \to \End(\scV)
\end{equation}
be the universal morphism
from the path algebra $\bC \Gamma$
of the quiver with relations associated with a dimer model.
This map is not an isomorphism in general,
and it is easy to see that
the injectivity of this map is equivalent to
the {\em first consistency condition}
of Mozgovoy and Reineke \cite{Mozgovoy-Reineke}.
The path algebra $\bC \Gamma$ is a Calabi-Yau algebra
of dimension three in the sense of Ginzburg \cite{Ginzburg_CYA}
if the dimer model satisfies the first consistency condition
\cite{Mozgovoy-Reineke, Davison, Broomhead}.
This in turn implies
\cite{Bridgeland-King-Reid, Van_den_Bergh_NCR}
that $\phi$ is an isomorphism,
the functor
$$
 \RHom(\scV, \bullet) : D^b \coh \scM_\theta \to D^b \module \bC \Gamma
$$
is an equivalence of triangulated categories, and
$
 \bC \Gamma
$
is a non-commutative crepant resolution of
a Gorenstein affine toric 3-fold.

The first consistency condition is an algebraic condition,
which is not easy to check in examples.
In this paper, we show that
a more tractable condition,
given in Definition \ref{df:consistency},
is equivalent to the first consistency condition
under the non-degeneracy assumption:

\begin{thm} \label{th:main}
For a non-degenerate dimer model,
\begin{itemize}
\item the first consistency condition,
\item the consistency condition in Definition \ref{df:consistency}, and
\item the {\em properly-orderedness} in the sense of Gulotta \cite{Gulotta}
\end{itemize}
are equivalent.
\end{thm}

It is known that the consistency condition
in Definition \ref{df:consistency}
implies the non-degeneracy
\cite[Proposition 6.2]{Ishii-Ueda_DMSMCv1}.
Together with a work of Kenyon and Schlenker
\cite[Theorem 5.1]{Kenyon-Schlenker},
Theorem \ref{th:main} implies
a result of Broomhead \cite{Broomhead}
that an {\em isoradial} dimer model satisfies the first consistency condition.
Here we note that isoradiality is a strong condition,
and a large number of
otherwise well-behaved dimer models
fall out of this class.

We recall basic definitions on dimer models
in Section \ref{sc:definitions}.
The content of Section \ref{sc:consistency}
has bubbled off from \cite[Section 5]{Ishii-Ueda_DMSMCv1},
and the rest of \cite{Ishii-Ueda_DMSMCv1} will appear
in a separate paper.
In Section \ref{sc:properly-ordered},
we show that a dimer model satisfies the consistency condition
in Definition \ref{df:consistency} if and only if
it is {properly-ordered} in the sense of Gulotta \cite{Gulotta}.
Relations between consistency conditions on dimer models
are also discussed by Bocklandt \cite[Section 8]{Bocklandt_CYAWQP}.

{\bf Acknowledgment}:
We thank Alastair Craw,
Nathan Broomhead,
Ben Davison,
Dominic Joyce,
Alastair King,
Diane Maclagan,
Balazs Szendroi,
Yukinobu Toda,
Michael Wemyss and
Masahito Yamazaki
for valuable discussions.
We also thank the anonymous referee
for carefully reading the manuscript and
suggesting a number of improvements.
A.~I. is supported
by Grant-in-Aid for Scientific Research (No.18540034).
K.~U. is supported
by Grant-in-Aid for Young Scientists (No.20740037)
and Engineering and Physical Sciences
Research Council (EP/F055366/1).

\section{Dimer models and quivers}
 \label{sc:definitions}

Let $T = \bR^2 / \bZ^2$
be a real two-torus
equipped with an orientation.
A {\em bicolored graph} on $T$
consists of
\begin{itemize}
 \item a finite set $B \subset T$ of black nodes,
 \item a finite set $W \subset T$ of white nodes, and
 \item a finite set $E$ of edges,
       consisting of embedded closed intervals $e$ on $T$
       such that one boundary of $e$ belongs to $B$
       and the other boundary belongs to $W$.
       We assume that two edges intersect
       only at the boundaries.
\end{itemize}
A {\em face} of a graph is a connected component
of $T \setminus \cup_{e \in E} e$.
A bicolored graph $G$ on $T$ is called a {\em dimer model}
if every face is simply-connected.

A {\em quiver} consists of
\begin{itemize}
 \item a set $V$ of vertices,
 \item a set $A$ of arrows, and
 \item two maps $s, t: A \to V$ from $A$ to $V$.
\end{itemize}
For an arrow $a \in A$,
$s(a)$ and $t(a)$
are said to be the {\em source}
and the {\em target} of $a$
respectively.
A {\em path} on a quiver
is an ordered set of arrows
$(a_n, a_{n-1}, \dots, a_{1})$
such that $s(a_{i+1}) = t(a_i)$
for $i=1, \dots, n-1$.
We also allow for a path of length zero,
starting and ending at the same vertex.
The {\em path algebra} $\bC Q$
of a quiver $Q = (V, A, s, t)$
is the algebra
spanned by the set of paths
as a vector space,
and the multiplication is defined
by the concatenation of paths;
$$
 (b_m, \dots, b_1) \cdot (a_n, \dots, a_1)
  = \begin{cases}
     (b_m, \dots, b_1, a_n, \dots, a_1) & s(b_1) = t(a_n), \\
      0 & \text{otherwise}.
    \end{cases}
$$
A {\em quiver with relations}
is a pair of a quiver
and a two-sided ideal $\scI$
of its path algebra.
For a quiver $\Gamma = (Q, \scI)$
with relations,
its path algebra $\bC \Gamma$ is defined as
the quotient algebra $\bC Q / \scI$.
Two paths $a$ and $b$ are said to be {\em equivalent}
if they give the same element in $\bC \Gamma$.

A dimer model $(B, W, E)$ encodes
the information of a quiver
$\Gamma = (V, A, s, t, \scI)$
with relations
in the following way:
The set $V$ of vertices
is the set of connected components
of the complement
$
 T \setminus (\bigcup_{e \in E} e),
$
and
the set $A$ of arrows
is the set $E$ of edges of the graph.
The directions of the arrows are determined
by the colors of the nodes of the graph,
so that the white node $w \in W$ is on the right
of the arrow.
In other words,
the quiver is the dual graph of the dimer model
equipped with an orientation given by
rotating the white-to-black flow on the edges of the dimer model
by minus 90 degrees.

The relations of the quiver are described as follows:
For an arrow $a \in A$,
there exist two paths $p_+(a)$
and $p_-(a)$
from $t(a)$ to $s(a)$,
the former going around the white node
connected to $a \in E = A$ clockwise
and the latter going around the black node
connected to $a$ counterclockwise.
Then the ideal $\scI$
of the path algebra is
generated by $p_+(a) - p_-(a)$
for all $a \in A$.

A {\em small cycle} on a quiver coming from a dimer model
is the product of arrows surrounding only a single node
of the dimer model.
A path $p$ is said to be {\em minimal}
if it is not equivalent to a path containing a small cycle.
A path $p$ is said to be {\em minimum}
if any path from $s(p)$ to $t(p)$ homotopic to $p$ is equivalent
to the product of $p$ and a power of a small cycle.
For a pair of vertices of the quiver,
a minimum path from one vertex to another may not exist,
and will always be minimal when it exists.

Small cycles starting from a fixed vertex are equivalent to each other.
Hence the sum $\omega$ of small cycles
over the set of vertices is a well-defined
element of the path algebra.
For any arrow $a$,
the small cycles $\omega_{s(a)}$ and $\omega_{t(a)}$
starting from the source $s(a)$ and the target $t(a)$
of $a$ respectively satisfies
$$
 a \omega_{s(a)} = \omega_{t(a)} a.
$$
If follows that $\omega$ belongs to the center of the path algebra,
and there is the universal map
$$
 \bC \Gamma \to \bC \Gamma[\omega^{-1}]
$$
into the localization of the path algebra
by the multiplicative subset generated by $\omega$.
Two paths $a$ and $b$ are said to be {\em weakly equivalent}
if they are mapped to the same element
in $\bC \Gamma[\omega^{-1}]$,
i.e., there is an integer $i \ge 0$ such that
$a \omega^i = b \omega^i$ in $\bC\Gamma$.
Note that
the following holds for the paths of the quiver.

\begin{lmm}
For two paths $a$ and $b$ with the same source and target, the following are equivalent.
\begin{itemize}
\item $a$ and $b$ are homotopy equivalent.
\item There are integers $i, j \ge 0$ such that $a \omega^i = b \omega^j$ in $\bC\Gamma$.
\item There is an integer $i \ge 0$ such that either $(a, b \omega^i)$ or $(a \omega^i, b)$ is
a weakly equivalent pair.
\end{itemize}
\end{lmm}

For example,
the paths $p$ and $q$
shown in Figure \ref{fg:inconsistent_path1}
are weakly equivalent,
but not equivalent.
They are homotopic and one has
$$
 \omega p = \omega q.
$$

\begin{figure}[htbp]
\centering
\input{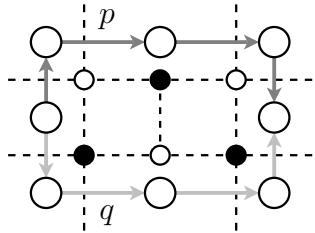}
\caption{A pair of weakly equivalent paths
which are not equivalent}
\label{fg:inconsistent_path1}
\end{figure}

A {\em perfect matching}
(or a {\em dimer configuration})
on a bicolored graph $G = (B, W, E)$
is a subset $D$ of $E$
such that for any node $v \in B \cup W$,
there is a unique edge $e \in D$
connected to $v$.
A dimer model is said to be {\em non-degenerate}
if for any edge $e \in E$,
there is a perfect matching $D$
such that $e \in D$.

A {\em zigzag path} is a path on a dimer model
which makes a maximum turn to the right on a white node
and to the left on a black node.
Note that it is not a path on a quiver.
We assume that a zigzag path does not have an endpoint,
so that we can regard a zigzag path
as a sequence $(e_i)$ of edges $e_i$ parameterized by $i \in \bZ$,
up to translations of $i$.
Figure \ref{fg:zigzag} shows an example
of a part of a dimer model
and a zigzag path on it.

\begin{figure}[htbp]
\centering
\input{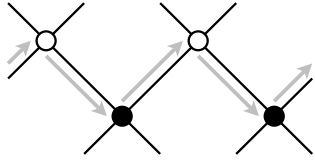}
\caption{A zigzag path}
\label{fg:zigzag}
\end{figure}

%

\begin{figure}
\centering
\begin{minipage}{.4 \linewidth}
\centering
\input{conifold_bt.pst}
\caption{A dimer model}
\label{fg:conifold_bt}
\end{minipage}
\begin{minipage}{.45 \linewidth}
\centering
\input{conifold_quiver.pst}
\caption{The corresponding quiver}
\label{fg:conifold_quiver}
\end{minipage}
\end{figure}

\begin{figure}
\begin{minipage}{\linewidth}
\centering
\input{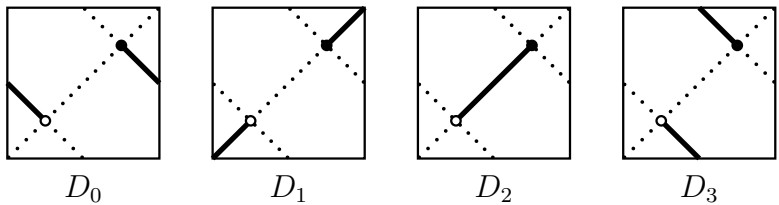}
\caption{Four perfect matchings}
\label{fg:conifold_dimers}
\end{minipage}
\end{figure}

%
%

As an example, consider the dimer model
in Figure \ref{fg:conifold_bt}.
The corresponding quiver is shown in Figure \ref{fg:conifold_quiver},
whose relations are given by
$$
 \scI = (d b c - c b d, d a c - c a d, a d b - b d a, a c b - b c a).
$$
This dimer model is non-degenerate,
and has four perfect matchings $D_0, \dots, D_3$
shown in Figure \ref{fg:conifold_dimers}.

We end this section with the following lemma:

\begin{lmm} \label{lm:exists_minimal_path}
Assume that a dimer model has a perfect matching $D$.
Then for any path $p$ on the quiver,
there are another path $q$ and
a non-negative integer $k$
such that $p$ is equivalent to $q \omega^k$ and
$q$ is not equivalent to a path containing a small cycle.
\end{lmm}

\begin{proof}
Consider the number of times the path $p$ crosses $D$.
Then this is a non-negative integer
which decreases by one
as one removes a small cycle from the path.
\end{proof}

The statement of Lemma \ref{lm:exists_minimal_path} can be false
if there is no perfect matching:
Figure \ref{fg:no-perfect-matching}
shows an example of a dimer model
without any perfect matching,
which we learned from Broomhead and King.
One can see from the relation
$$
 a = e a d c b
$$
that
$$
 c b f e a d
  = c b f e^2 a d c (b d)
  = c b f e^3 a d c (b d) c (b d)
  = \cdots,
$$
which shows that the loop
$c b f e a d$ can be divided by any power of the small cycle
$b d$.

\begin{figure}[htbp]
\centering
\input{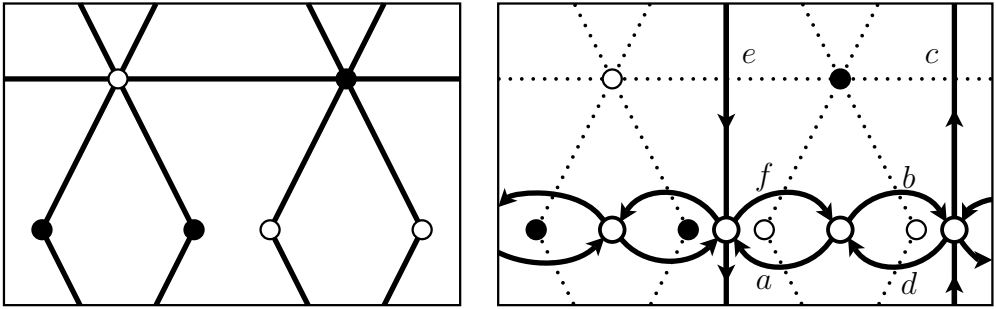}
\caption{A dimer model without any perfect matching}
\label{fg:no-perfect-matching}
\end{figure}

\section{Consistency conditions for dimer models}
 \label{sc:consistency}

The following notion is due to
Duffin \cite{Duffin} and Mercat \cite{Mercat_DRSIM}:

\begin{dfn} \label{df:isoradial}
A dimer model is {\em isoradial}
if one can choose an embedding of the graph into the torus
so that every face of the graph is a polygon
inscribed in a circle of a fixed radius
with respect to a flat metric on the torus.
Here, the circumcenter of any face must be contained
in the face.
\end{dfn}

A dimer model is isoradial
if and only if zigzag paths behave like straight lines:

\begin{thm}[{Kenyon and Schlenker \cite[Theorem 5.1]{Kenyon-Schlenker}}]
 \label{th:Kenyon-Schlenker}
A dimer model is isoradial
if and only if the following conditions are satisfied:
\begin{enumerate}
 \item
Every zigzag path is a simple closed curve.
 \item
The lift of any pair of zigzag paths to the universal cover
of the torus share at most one edge.
\end{enumerate}
\end{thm}

The following condition is introduced by Mozgovoy and Reineke:

\begin{dfn}[{\cite[Condition 4.12]{Mozgovoy-Reineke}}] \label{df:MR}
A dimer model is said to satisfy the {\em first consistency condition}
if weakly equivalent paths are equivalent.
\end{dfn}


We regard a zigzag path on the universal cover
as a sequence $(e_i)$ of edges $e_i$ parameterized by $i \in \bZ$,
up to translations of $i$.

\begin{dfn}\label{df:intersection}
Let $z=(e_i)$ and $w=(f_i)$ be two zigzag paths on the universal cover.
We say that $z$ and $w$ intersect
if there are $i, j \in \bZ$ with $e_i=f_j$
such that if $u, v$ are the maximum and the minimum of $t$
with $e_{i+t}=f_{j-t}$ respectively,
then $u-v \in 2\bZ$.
In this case, the sequence
$
 ( e_{i+v} = f_{j-v},
   e_{i+v+1} = f_{j-v-1}, \dots,
   e_{i+u} = f_{j-u})
$
of intersections is counted as a single intersection.
We say that $z$ has a self-intersection if there is a pair $i \ne j$
with $e_i = e_j$ such that
the directions of $z$ at $e_i$ and $e_j$ are opposite,
and $u - v \in 2\bZ$ for $u$ and $v$ defined similarly as above.
We say that $z$ is homologically trivial
if the map $i \mapsto e_i$ is periodic.
\end{dfn}
Note that if $u-v>0$ in the above definition,
then the nodes between $e_v$ and $e_u$ are divalent.
According to this definition,
there are cases where $z$ and $w$ have a common nodes or common edges,
but they do not intersect
as shown in Figure \ref{fg:zigzag-intersection}.
The assumption $u - v \in 2 \bZ$ is needed
to remove the effect of a divalent node;
if there is no divalent node,
then a pair of zigzag paths intersect
if and only if they have a common edge.

\begin{figure}[htbp]
\centering
\input{zigzag-intersection.pst}
\caption{Examples of an intersection (left) and
a non-intersection (right)}
\label{fg:zigzag-intersection}
\end{figure}

The following condition is slightly weaker than
isoradiality, and easy to check in examples:

\begin{dfn} \label{df:consistency}
A dimer model is said to be {\em consistent} if
\begin{itemize}
 \item
there is no homologically trivial zigzag path,
 \item
no zigzag path has a self-intersection
on the universal cover, and
 \item
no pair of zigzag paths intersect each other on the universal cover
in the same direction more than once.
\end{itemize}
\end{dfn}

Here, the third condition means that
if a pair $(z, w)$ of zigzag paths has two intersections $a$ and $b$
and the zigzag path $z$ points from $a$ to $b$,
then the other zigzag path $w$ must point from $b$ to $a$. 

\begin{figure}[htbp]
\begin{minipage}{.45 \linewidth}
\centering
\input{inconsistent-2-zigzag.pst}
\caption{A homologically trivial zigzag path}
\label{fg:inconsistent-2-zigzag}
\end{minipage}
\begin{minipage}{.45 \linewidth}
\centering
\input{inconsistent.pst}
\caption{An inconsistent dimer model}
\label{fg:inconsistent}
\end{minipage}
\centering
\input{inconsistent_zigzag_1.pst}
\caption{A pair of zigzag paths in the same direction
intersecting twice}
\label{fg:inconsistent_zigzag}
\end{figure}

Figure \ref{fg:inconsistent-2-zigzag} shows a part of
an inconsistent dimer model
which contains a homologically trivial zigzag path.
Figure \ref{fg:inconsistent} shows an inconsistent dimer model,
which contains a pair of zigzag paths
intersecting in the same direction twice
as in Figure \ref{fg:inconsistent_zigzag}.

On the other hand,
a pair of zigzag paths going in the opposite direction
may intersect twice in a consistent dimer model.
Figure \ref{fg:P1P1_II_zigzag} shows a pair of such zigzag paths
on a consistent dimer model
in Figure \ref{fg:P1P1_II}.

\begin{figure}[htbp]
\begin{minipage}{\linewidth}
\centering
\input{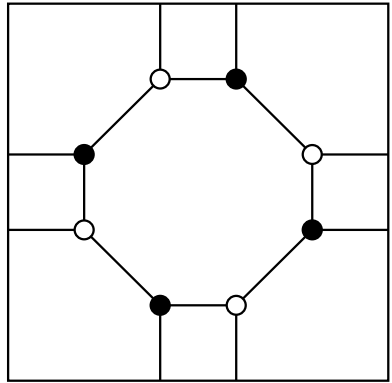}
\caption{A consistent non-isoradial dimer model}
\label{fg:P1P1_II}
\end{minipage}
\end{figure}

\begin{figure}[htbp]
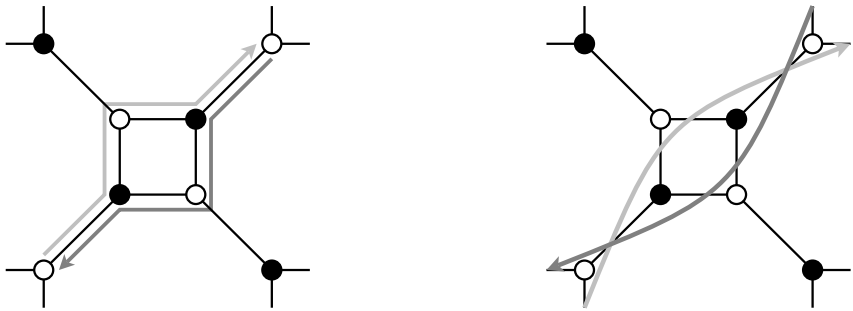

\begin{minipage}{.45 \linewidth}
\centering
\input{consistent_1.pst}
\end{minipage}
\begin{minipage}{.45 \linewidth}
\centering
\input{consistent_2.pst}
\end{minipage}
\caption{A pair of zigzag paths in the opposite direction
intersecting twice}
\label{fg:P1P1_II_zigzag}
\end{figure}

To obtain a criterion for the minimality of a path,
we discuss the intersection of a path of the quiver and a zigzag path.
Note that paths of the quiver and zigzag paths are both regarded
as sequences of arrows of the quiver,
where the former are finite and the latter are infinite.

\begin{dfn}
Let $p=a_1 a_2 \dots$ be a path of the quiver ($a_i \in A$) and
$z= (b_i)_{i\in \bZ}$ be a zigzag path.
We say $p$ intersects $z$ at an arrow $a$
if there are $i, j$ with $a=a_i=b_j \in A = E$,
satisfying the following condition:
If $u, v$ denote the maximum and the minimum of $t$
with $a_{i+t}=b_{j-t}$ respectively,
then $u-v$ is even.
In this case, the sequence $(a_{i+v} = b_{j-v}, \dots, a_{i+u}=b_{j-u})$
is counted as a single intersection.
\end{dfn}
Figure \ref{fg:path-zigzag-non-intersection} shows
an example of a non-intersection; the path shown in dark gray
does not intersect the zigzag path shown in light gray.
Note that the dark gray path is equivalent to the dashed path,
which does not have a common edge (or an arrow)
with the light gray path.
\begin{figure}[htbp]
\centering
\input{path-zigzag-non-intersection.pst}
\caption{An example of a non-intersection}
\label{fg:path-zigzag-non-intersection}
\end{figure}

The following lemma is obvious from the definition
of the equivalence relations of paths:


\begin{lmm} \label{lm:zigzag_twice}
Let $z$ be a zigzag path on the universal cover.
Suppose that a path $p'$ is obtained from another path $p$
by replacing $p_+(a) \subset p$ with $p_-(a)$
or the other way around
for a single arrow $a$,
as in the definition of the equivalence relations of paths.
If neither $ap_+(a)$ nor $p_+(a)a$ is a part of $p$,
then there is a natural bijection between
the intersections of $z$ and $p$ and those of $z$ and $p'$.
If $a$ is not a part of $p$,
then this bijection preserves the order of intersections
along $z$.
\end{lmm}

Because $a p_+(a)$ and $p_+(a) a$ are small cycles,
the first half of Lemma \ref{lm:zigzag_twice}
immediately gives the following:

\begin{crl} \label{cr:zigzag-loop}
A minimal path which does not intersect a zigzag path $z$
cannot be equivalent to a path intersecting $z$.
\end{crl}

Lemma \ref{lm:zigzag_twice} also gives the following:

\begin{crl} \label{cr:minimal_path}
Let $p$ be a path of the quiver.
If there is no zigzag path that intersects $p$
more than once in the same direction on the universal cover,
then $p$ is minimal.
\end{crl}
\begin{proof}
Assume that there is no zigzag path that intersects $p$
more than once in the same direction on the universal cover.
If $p$ contains an arrow $a$ and
either $p_+(a)$ or $p_-(a)$,
then one of two zig-zag paths
containing the edge corresponding to $a$
intersects $p$ more than once in the same direction
on the universal cover.
It follows that
if $p$ contains $p_+(a)$ or $p_-(a)$
for an arrow $a$, then $p$ does not contain $a$.
Let $p'$ be a path related to $p$ as in Lemma \ref{lm:zigzag_twice}.
Since $p$ does not contain small cycles $ap_+(a)$ or $p_+(a)a$,
Lemma \ref{lm:zigzag_twice} implies that
$p'$ also satisfies the assumption and
hence does not contain a small cycle.
By repeating this argument,
we can see that if a path is equivalent to $p$,
then it does not contain a small cycle.
\end{proof}

The following lemma shows that
the consistency condition implies the first consistency condition
of Mozgovoy and Reineke:

\begin{lmm}\label{lm:consistency_implies_first_consistency}
If weak equivalence does not imply equivalence,
then the dimer model is not consistent.
\end{lmm}

\begin{proof}
Assume for contradiction that a consistent dimer model
has a pair of weakly equivalent paths
which are not equivalent.
Then there is a pair $(a, b)$ of paths
on the universaly cover such that
\begin{itemize}
\item There is an integer $i \ge 0$ such that either $(a,b\omega^i)$ or $(a\omega^i, b)$ is weakly equivalent but not equivalent. 
\item If one of $a$ and $b$ contains loops, then it is a loop and the other one is a trivial path.
\item $a$ and $b$ meet only at the endpoints.
\end{itemize}
Choose one of such pairs,
without fixing the endpoints,
so that the area bounded by $a$ and $b$ is minimal
with respect to the inclusion relation.

Figure \ref{fg:consistency1-proof-1} shows
a pair $(a, b)$ of such paths.
We may assume that $a$ is a non-trivial path.
\begin{figure}[htbp]
\centering
\input{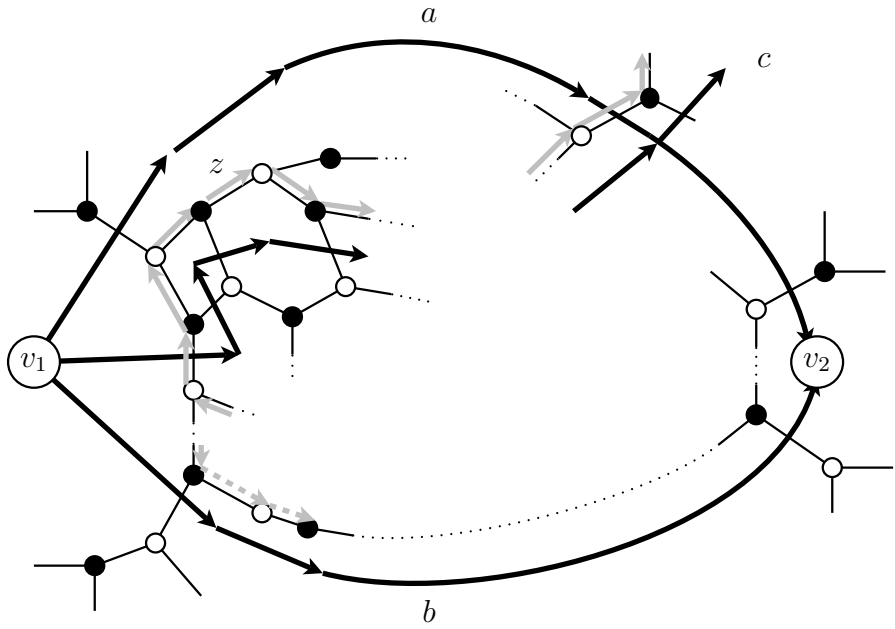}
\caption{A pair of inequivalent paths which are weakly equivalent}
\label{fg:consistency1-proof-1}
\end{figure}
Let $v_1$ and $v_2$ be the source and the target of $a$ respectively.
To show the inconsistency of the dimer model,
consider the zigzag path $z$
which starts from the white node
just on the right of the first arrow in the path $a$
as shown in light gray in Figure \ref{fg:consistency1-proof-1}.

We show that if $z$ crosses $a$,
then it contradicts the minimality of the area.
Assume that $z$ crosses $a$,
and consider the path $c$ which goes along $z$
as in Figure \ref{fg:consistency1-proof-1}.
Since $z$ crosses $a$,
the path $c$ also crosses $a$.
Let $v_3$ be the vertex where $a$ and $c$ intersects,
and $a'$ and $c'$ be the parts of $a$ and $c$
from $v_1$ to $v_3$ respectively.
The part of $a$ from $v_3$ to $v_2$ will be denoted by $d$
as in Figure \ref{fg:consistency1-proof-2}.
\begin{figure}[htbp]
\centering
\input{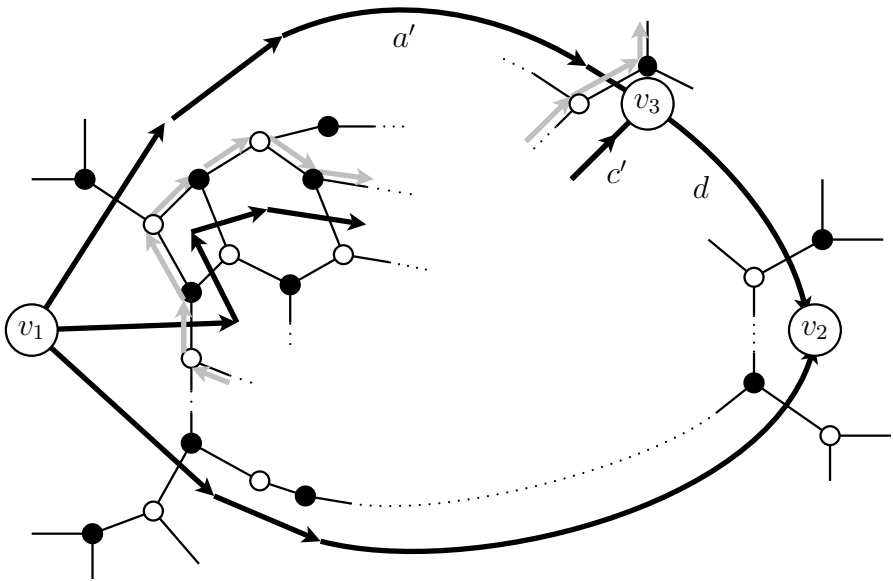}
\caption{The paths $a'$ and $c'$}
\label{fg:consistency1-proof-2}
\end{figure}

If there is a zigzag path $w$ which intersects $c'$
more than once in the same direction,
then $w$ also intersects $z$ more than once in the same direction,
which contradicts the assumption
that the dimer model is consistent.
Hence no zigzag path intersects $c'$
more than once in the same direction,
so that $c'$ is minimal
by Corollary \ref{cr:minimal_path}.

Suppose $d c'$ is different from $b$.
Then by the minimality of the area and the minimality of $c'$,
there are non-negative integers $i$ and $j$
such that $a'$ is equivalent to $c'\omega^i$ and
either $(dc'\omega^j, b)$ or $(dc', b\omega^j)$ are equivalent pairs.
Then one of $(a, b\omega^{i-j})$, $(a\omega^{j-i}, b)$ and
$(a, b\omega^{i+j})$ is an equivalent pair,
which contradicts the assumption.
If $d c'$ coincides with $b$,
then $b$ is equivalent to a path that goes along the opposite side of $z$
as in Figure \ref{fg:consistency1-proof-3},
which contradicts the minimality of the area.
\begin{figure}[htbp]
\centering
\input{consistency1-proof-3.pst}
\caption{A path equivalent to $d c'$}
\label{fg:consistency1-proof-3}
\end{figure}


Hence the zigzag path $z$ cannot cross the path $a$.
In the same way,
the dashed gray zigzag path in Figure \ref{fg:consistency1-proof-1}
cannot cross the path $b$.
It follows that if we extend these two zigzag paths in both directions,
then they will intersect in the same direction more than once
or have a self-intersection.
This contradicts the consistency of the dimer model,
and Lemma \ref{lm:consistency_implies_first_consistency} is proved.

\end{proof}

\begin{lmm}
For a path $p$ in a consistent dimer model,
the following are equivalent:
\begin{enumerate}
 \item \label{it:minimal}
$p$ is minimal.
 \item \label{it:minimum}
$p$ is minimum.
 \item \label{it:zigzag}
There is no zigzag path
that intersects $p$ more than once
in the same direction on the universal cover.
\end{enumerate}
\end{lmm}
\begin{proof}
It is clear that \ref{it:minimum} implies \ref{it:minimal}.
To show the converse,
take a minimal path $p$ and
a path $q$ from $s(p)$ to $t(p)$ homotopic to $p$.
Then $(p, q \omega^i)$ or $(p \omega^i, q)$ is weakly equivalent,
hence equivalent.
By the minimality of $p$,
$p\omega^i$ is equivalent to $q$,
which means $p$ is minimum.

Corollary \ref{cr:minimal_path} states that
\ref{it:zigzag} implies \ref{it:minimal}.
To show the converse,
suppose there is a zigzag path $z$ as above.
Let $a_1$ and $a_2$ be arrows on the intersection of $z$ and $p$
such that the directions are from $a_1$ to $a_2$ on both $z$ and $p$,
and their parts between $a_1$ and $a_2$ do not meet each other.
Let $p'$ be the part of $p$ from $s(a_1)$ to $t(a_2)$.
There is a path $q$ from $s(a_1)$ to $t(a_2)$ which is parallel to $z$.
Since $q$ is minimal by Corollary \ref{cr:minimal_path},
it is minimum and
there is an integer $i \ge 0$ such that
$p'$ is equivalent to $q\omega^i$.
If $p'$ is also minimal, $i$ must be zero and therefore $p'$ is equivalent to $q$. 
This contradicts Lemma \ref{lm:zigzag_twice}
and thus $p$ is not minimal.
\end{proof}

The following lemmas show that
the first consistency condition of Mozgovoy and Reineke
together with the existence of a perfect matching
implies the consistency condition:

\begin{lmm} \label{lm:consistency2-proof-1}
Assume that a dimer model has a perfect matching
and a pair of zigzag paths
intersecting in the same direction twice on the universal cover,
none of which has a self-intersection.
Then there is a pair of inequivalent paths
which are weakly equivalent.
\end{lmm}

\begin{proof}
For a pair $(z, w)$ of zigzag paths
intersecting in the same direction twice
on the universal cover,
consider the pair $(a, b)$ of paths as shown in dark gray
in Figure \ref{fg:consistency2-proof-1}.
\begin{figure}[htbp]
\centering
\input{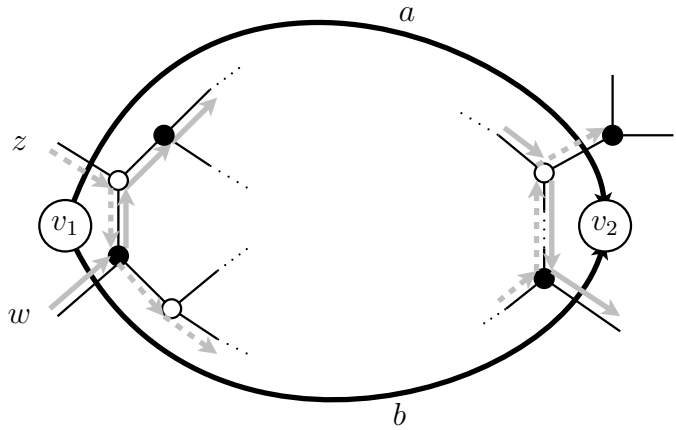}
\caption{A pair of inequivalent paths which are weakly equivalent}
\label{fg:consistency2-proof-1}
\end{figure}
Our assumption that $w$ does not have a self-intersection
implies that $a$ does not intersect $w$.
We claim that there is a minimal path $a'$
which does not intersect $w$
such that $a = a' \omega^k$ for some $k \in \bN$.
The existence of such $a'$ and $k$
follows from Corollary \ref{cr:zigzag-loop}
and the existence of a perfect matching:
A perfect matching intersects $a$ in a finite number of points,
and the number of intersection decreases by one
as one factors out a small cycle.
Hence the process of
\begin{itemize}
 \item
deforming the path without letting it
intersect $w$ (Lemma \ref{lm:zigzag_twice}), and
 \item
factoring out a small cycle if any
\end{itemize}
must terminate in finite steps.
Moreover,
the resulting path $a'$ cannot be equivalent to a path
intersecting $w$ by Corollary \ref{cr:zigzag-loop}.
Similarly, there is a minimal path $b'$ from $v_1$ to $v_2$
which does not intersect $z$.
On the other hand,
$a'$ and $b'$ intersect $z$ and $w$ respectively
for topological reason.
It follows that
$(a', b' \omega^i)$ or $(a' \omega^i, b')$
for some non-negative integer $i$
gives a pair of weakly equivalent paths
which are not equivalent.
\end{proof}

\begin{lmm} \label{lm:zigzag-loop}
Assume that a dimer model has a perfect matching
and a zigzag path with a self-intersection
on the universal cover.
Then there is a pair of inequivalent paths
which are weakly equivalent.
\end{lmm}
\begin{proof}
Let $z$ be a zigzag path on the universal cover
with a self-intersection and
$e_0 e_1 e_2 \dots e_n e_0$ be a loop in $z$,
where $z$ has a self-intersection at $e_0$
and does not have any self-intersection
in $(e_1, \ldots, e_n)$.
The union of the edges $e_1, \dots, e_n$ will be denoted by $C$.
\begin{figure}[htbp]
\begin{minipage}{.45 \linewidth}
\centering
\input{consistency3-proof-1.pst}
\caption{A pair of inequivalent paths which are weakly equivalent}
\label{fg:consistency3-proof-1}
\end{minipage}
\begin{minipage}{.45 \linewidth}
\centering
\input{inconsistent-2-path.pst}
\caption{Homologically trivial zigzag path
and a cyclic path on the quiver}
\label{fg:inconsistent-2-path}
\end{minipage}

\end{figure}

Regarding $e_0$ as an arrow,
we put $v_1=s(e_0)$ and $v_2=t(e_0)$.
There is a path $b$ from $v_1$ to $v_2$ which goes along $z$.
The edge $e_0$ as an arrow of the quiver
also forms a path from $v_1$ to $v_2$.
We show that the path $e_0$ is minimal, and
\begin{itemize}
 \item
there is a minimal path $b'$ from $v_1$ to $v_2$
which is not equivalent to $e_0$, or
 \item
there is a non-trivial cyclic path
which is not equivalent to any positive power of a small cycle.
\end{itemize}
In the latter case,
since we are working on the universal cover,
this cyclic path is homologically trivial, and
the pair consisting of this cyclic path
and a suitable power of a small cycle
gives a pair of inequivalent paths
which are weakly equivalent.
In the former case,
there is a non-negative integer $i$ such that
either $(e_0, b' \omega^i)$ or $(e_0 \omega^i, b')$
is a pair of weakly equivalent paths,
since both $e_0$ and $b'$ are paths from $v_1$ to $v_2$
on the universal cover, and hence homotopic.
This pair of paths cannot be equivalent
since $e_0$ and $b'$ are minimal.

To obtain a minimal path from $b$,
we first remove as many small cycles from $b$ as possible
without making it intersect $C$.
This process terminates in finite steps
just as in the proof of Lemma \ref{lm:consistency2-proof-1}.
The resulting path $b_1$ may not be minimal yet
since it might allow a deformation
first to a path intersecting $C$
and then to a path containing small cycles.
Assume that another path $b_1'$ from $v_1$ to $v_2$
intersecting $C$
is obtained from $b_1$
by replacing $p_{-}(a) \subset b_1$ with $p_+(a)$
(or the other way around, depending
on the color of the node at $e_0 \cap e_1$)
for a single arrow $a$.
Since $C$ is a part of a zigzag path,
it follows, from the definitions of a zigzag path
and the equivalence of paths
just as in Lemma \ref{lm:zigzag_twice},
that the arrow $a$ must be $e_0$.
(Lemma \ref{lm:zigzag_twice} roughly
states that one needs a small cycle
to deform a path across a zigzag path.
Since $b_1$ does not contain a small cycle,
the only way to deform it across $C$ is to
deform it by the equivalence relation at $e_0$.
Unfortunately, one cannot apply Lemma \ref{lm:zigzag_twice}
directly in the present situation
since $C$ may intersect $z \setminus C$.)

Thus $b_1$ contains $p_{-}(e_0)$ (or $p_{+}(e_0)$)
and is written as $b_1 = c p_{-}(e_0) d$
(or $b_1 = c p_{+}(e_0) d$),
where $c$ and $d$ are paths from $v_1$ to $v_2$.
At least one of them (say, $c$) is not homotopic to
the arrow $e_0$
in $\bR^2 \setminus C$.
Take a perfect matching $D$ and count the number
$|c \cap D|$ of edges of $D$
which meet the path $c$.

If the number $|c \cap D|$ is equal to $|b_1 \cap D|$,
then $p_{-}(e_0) d$ is a non-trivial cyclic path on the quiver
which does not meet $D$ at all.
Note that for any given perfect matching,
equivalent paths have
the same numbers of arrows meeting that perfect matching.
Since a small cycle meet any perfect matching
at exactly one edge,
the cyclic path $p_{-}(e_0) d$ cannot be equivalent
to any positive power of a small cycle.

If the number $|c \cap D|$ is smaller than $|b_1 \cap D|$,
then we set $b_2 = c$ and repeat this process.
After finitely many steps,
we obtain a path $b' = b_n$ such that
\begin{itemize}
 \item
$b'$ is not homotopic to $e_0$ in $\bR^2 \setminus C$, and
 \item
$b'$ is not equivalent to a path 
containing a small cycle,
so that $b'$ is minimal,
\end{itemize}
or a cyclic path which is not equivalent
to any positive power of a small cycle.

To show that the path $e_0$ is minimal,
note that
the arrow $e_0$ can be equivalent to another path
only if the edge $e_0$ is the first
of several consecutive edges
connected by divalent nodes.
Since $z$ has a self-intersection at $e_0$,
the number of consecutive edges connected
by divalent nodes must be odd and
$e_0$ can be equivalent only to arrows.
This shows that the path $e_0$ is minimal.

It is clear that $b'$ is a path of length
greater than one.
This shows that $e_0$ is not equivalent to $b'$,
and Lemma \ref{lm:zigzag-loop} is proved.
\end{proof}

The following lemma can be shown in an analogous way:

\begin{lmm}
Assume that a dimer model has a perfect matching
and a zigzag path with the trivial homology class,
then there is a cyclic path on the quiver
which is weakly equivalent to some power of a small cycle
but not equivalent.
\end{lmm}

Indeed, consider the path which goes around the zigzag path,
and factor out all the possible small cycles.
Then one ends up with
a path weakly equivalent to a power of a small cycle
but not equivalent to it.

For example,
the path on the quiver shown in Figure \ref{fg:inconsistent-2-path}
is weakly equivalent to a small cycle
as shown in Figure \ref{fg:inconsistent-2-loop},
although it is not equivalent;
if we call the idempotent element
in the path algebra
corresponding to the top-left vertex
and the path shown in Figure \ref{fg:inconsistent-2-path}
starting from the top-left vertex
as $e$ and $p$ respectively,
then one has
$
 p \ne e \omega
$
and
$
 p \omega = e \omega^2.
$

\begin{figure}[htbp]
\centering
\input{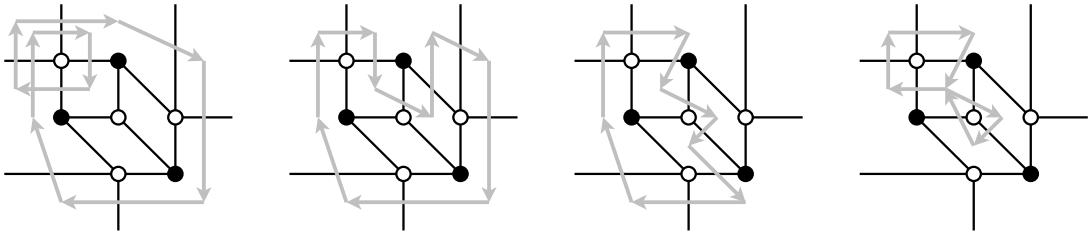}
\caption{Deforming a path on the quiver}
\label{fg:inconsistent-2-loop}
\end{figure}

%
%

\section{Properly-ordered dimer models}
 \label{sc:properly-ordered}

For a node in a dimer model,
the set of zigzag paths going through the edges adjacent to it
has a natural cyclic ordering
given by the directions of the outgoing paths
from the node.
On the other hand,
the homology classes of these zigzag paths determine
another cyclic ordering
if these classes are distinct.

\begin{dfn}[{Gulotta \cite[section 3.1]{Gulotta}}]
 \label{df:properly-ordered}
A dimer model is {\em properly ordered} if
\begin{enumerate}
\item
there is no homologically trivial zigzag path,
\item
no zigzag path has a self-intersection on the universal cover,
 \item
no pair of zigzag paths in the same homology class have a common node, and
 \item
for any node of the dimer model,
the cyclic order on the set of zigzag paths
going through that node
coincides with the cyclic order
determined by their homology classes.
\end{enumerate}
\end{dfn}
Here,
the homology group of the torus $T = \bR^2 / \bZ^2$
is identified with $\bZ^2$ in a natural way.
The {\em slope} of a zigzag path is
$$
 \frac{(u, v)}{\sqrt{u^2+v^2}} \in S^1,
$$
where $(u, v) \in \bZ^2$ is the homology class
of the zigzag path.
The lack of a self-intersection
implies that $(u, v)$ is a primitive element,
so that a set of zigzag paths with distinct homology classes
has a well-defined counter-clockwise cyclic order.

A consistent dimer model is properly ordered:

\begin{lmm} \label{lm:c-p}
In a consistent dimer model,
the cyclic order of the zigzag paths
around any node of the dimer model
is compatible with the cyclic order
determined by their slopes.
\end{lmm}

\begin{proof}
Let $z_1$, $z_2$ and $z_3$ be a triple of zigzag paths
passing through a node of the dimer model
along neighboring edges at the node
whose cyclic order around the node
does not respect the cyclic order
determined by their slopes.
Then two of them must intersect more than once
in the same direction on the universal cover.
\end{proof}

The converse is also true:

\begin{lmm} \label{lm:p-c}
A properly-ordered dimer model is consistent.
\end{lmm}

\begin{proof}
Assume for contradiction that
a properly-ordered dimer model has a pair
$z_1 = (e_k)_{k \in \bZ}$ and
$z_2 = (f_\ell)_{\ell \in \bZ}$
of zigzag paths
intersecting in the same direction more than once
on the universal cover.
We show that there is an infinite sequence
$(z_3, z_4, \ldots)$ of zigzag paths on the universal cover
with distinct slopes,
which contradicts the finiteness of the set of slopes.

An intersection
$
 (e_{i} = f_{j+u}, e_{i+1} = f_{j+u-1}, \ldots, e_{i+u}=f_{j})
$
of $z_1$ and $z_2$
where $i, j \in \bZ$ and $u \in 2 \bN$
is called a {\em last} intersection
if $(e_k)_{k>i+u}$ does not intersect $(f_\ell)_{\ell>j+u}$.
Another intersection
$
 (e_{i'} = f_{j'+u'}, e_{i'+1} = f_{j'+u'-1},
  \ldots, e_{i'+u'}=f_{j'})
$
for $i'+u'<i$
is called the {\em second last} intersection along $z_1$
if $(e_k)_{i'+u'<k<i}$
does not intersect $(f_\ell)_{\ell < j}$.
Although a last intersection may not be unique,
and not all last intersections may have the second last intersection,
the assumption that
$z_1$ and $z_2$ intersect in the same direction more than once
implies the existence of at least one last intersection
having the second last intersection.

Figure \ref{fg:p-c-proof} shows
a part of a pair of zigzag paths
near a last and the second last intersections.
We have suppresed the rest of the paths,
which may also intersect this part.
We choose the names $z_1$ and $z_2$
for these zigzag paths,
so that the node $e_{i+u} \cap e_{i+u+1}$
at the last intersection is white
as in Figure \ref{fg:p-c-proof1}.
Although the second last intersection in this figure
may be the one along $z_2$
instead of the one along $z_1$,
this does not affect the discussion below.

\begin{figure}[htbp]
\begin{minipage}{.45 \linewidth}
\centering
\input{p-c-proof.pst}
\caption{A pair of intersections of zigzag paths}
\label{fg:p-c-proof}
\end{minipage}
\begin{minipage}{.45 \linewidth}
\centering
\input{p-c-proof1.pst}
\caption{$z_3$ bending over to the left}
\label{fg:p-c-proof1}
\end{minipage}
\begin{minipage}{\linewidth}
\centering
\input{p-c-proof2.pst}
\caption{$z_3$ bending over to the right}
\label{fg:p-c-proof2}
\end{minipage}
\end{figure}

Now choose the third zigzag path
$
 z_3 = (g_m)_{m \in \bZ}
$
as the one going in the direction opposite to $z_2$
from the second last intersection
as shown in dotted arrow in Figure \ref{fg:p-c-proof1},
so that
$
 g_0 = f_{j'-1}.
$
Note that $z_2$ and $z_3$ may not intersect at
$
 g_0 = f_{j'-1}
$
if the node at $g_0 \cap g_1$ is divalent.
The cyclic order on the set of zigzag paths,
passing through the node
$g_{-1} \cap g_0$
where $z_1$, $z_2$ and $z_3$ meet,
is given by $(z_1, z_2, z_3, \cdots)$.
Since the dimer model is properly-ordered,
the slopes of $z_1$, $z_2$ and $z_3$ have
this cyclic order.
The slope of a zigzag path determines
the asymptotic behavior of the zigzag path
on the universal cover,
so that the zigzag paths $z_1$, $z_2$ and $z_3$
must have this cyclic order
outside of a compact set.
Combined with the assumption
that the intersection
$
 (e_{i} = f_{j+u}, e_{i+1} = f_{j+u-1}, \ldots, e_{i+u}=f_{j})
$
is a last intersection of $z_1$ and $z_2$,
this implies that
\begin{itemize}
 \item
$(g_m)_{m>0}$ intersects $(e_k)_{k>i'+u'}$, or
\item
$(g_m)_{m>0}$ intersects $(f_\ell)_{\ell>j'+u'}$.
\end{itemize}
Schematic pictures of examples of the former case
and the latter case are shown
in Figure \ref{fg:p-c-proof1} and
Figure \ref{fg:p-c-proof2}.
It may also happen that $(g_m)_{m>0}$ intersect
both $(e_k)_{k>i'+u'}$ and $(f_\ell)_{\ell>j'+u'}$.

In the former case,
the part $(g_{m})_{m >0}$ of the zigzag path $z_3$
intersects the zigzag path $z_1$
in the same direction more than once,
and one can find a pair of 
a last and the second last intersection
as in Figure \ref{fg:p-c-proof},
where the solid arrow represents $z_1$
and the gray arrow represents $z_3$ this time.
Now we can repeat the same argument to obtain
another zigzag path $z_4$ such that
\begin{itemize}
 \item
the cyclic order of the slopes is $(z_2, z_3, z_4, z_1)$, and
 \item
$z_4$ intersects $z_1$ or $z_3$
in the same direction more than once.
\end{itemize}

In the latter case,
the lack of self-intersection
of zigzag paths in a properly-ordered dimer model
implies that
the part $(g_m)_{m<0}$ of the zigzag path $z_3$
intersects the part $(f_\ell)_{\ell>j'+u'}$
of the zigzag path $z_2$,
and one can find
a pair of a last and the second last intersections
as in Figure \ref{fg:p-c-proof},
where the solid arrow represents $z_3$ and
the gray arrow represents $z_2$ this time.
Now we can repeat the same argument to obtain
another zigzag path $z_4$ such that
\begin{itemize}
 \item
the cyclic order of the slopes is $(z_2, z_4, z_3, z_1)$, and
 \item
$z_4$ intersects $z_2$ or $z_3$
in the same direction more than once.
\end{itemize}

In both cases,
we obtain a zigzag path $z_4$
whose slope is different from the slope
of any of $z_1$, $z_2$ and $z_3$.
By continuing this process,
we obtain an infinite sequence
$(z_5, z_6, \ldots)$ of zigzag paths
with distinct slopes,
and Lemma \ref{lm:p-c} is proved.

\end{proof}

By combining Lemma \ref{lm:c-p} with Lemma \ref{lm:p-c},
one obtains the equivalence
between consistency condition in Definition \ref{df:consistency}
and Gulotta's condition:

\begin{prp}
A dimer model is consistent
if and only if it is properly-ordered.
\end{prp}

\bibliographystyle{plain}
\bibliography{bibs}

\def\cprime{$'$}
\begin{thebibliography}{10}

\bibitem{Bocklandt_CYAWQP}
Raf Bocklandt.
\newblock {C}alabi {Y}au algebras and weighted quiver polyhedra.
\newblock arXiv:0905.0232.

\bibitem{Bridgeland-King-Reid}
Tom Bridgeland, Alastair King, and Miles Reid.
\newblock The {M}c{K}ay correspondence as an equivalence of derived categories.
\newblock {\em J. Amer. Math. Soc.}, 14(3):535--554 (electronic), 2001.

\bibitem{Broomhead}
Nathan Broomhead.
\newblock Dimer models and {C}alabi-{Y}au algebras.
\newblock arXiv:0901.4662.

\bibitem{Davison}
Ben Davison.
\newblock Consistency conditions for brane tilings.
\newblock arXiv:0812.4185.

\bibitem{Duffin}
R.~J. Duffin.
\newblock Potential theory on a rhombic lattice.
\newblock {\em J. Combinatorial Theory}, 5:258--272, 1968.

\bibitem{Franco-Hanany-Martelli-Sparks-Vegh-Wecht_GTTGBT}
Sebasti{\'a}n Franco, Amihay Hanany, Dario Martelli, James Sparks, David Vegh,
  and Brian Wecht.
\newblock Gauge theories from toric geometry and brane tilings.
\newblock {\em J. High Energy Phys.}, (1):128, 40 pp. (electronic), 2006.

\bibitem{Franco-Hanany-Vegh-Wecht-Kennaway_BDQGT}
Sebasti{\'a}n Franco, Amihay Hanany, David Vegh, Brian Wecht, and Kristian~D.
  Kennaway.
\newblock Brane dimers and quiver gauge theories.
\newblock {\em J. High Energy Phys.}, (1):096, 48 pp. (electronic), 2006.

\bibitem{Franco-Vegh_MSGTDM}
Sebasti{\'a}n Franco and David Vegh.
\newblock Moduli spaces of gauge theories from dimer models: proof of the
  correspondence.
\newblock {\em J. High Energy Phys.}, (11):054, 26 pp. (electronic), 2006.

\bibitem{Ginzburg_CYA}
Victor Ginzburg.
\newblock {C}alabi-{Y}au algebras.
\newblock math.AG/0612139, 2006.

\bibitem{Gulotta}
Daniel~R. Gulotta.
\newblock Properly ordered dimers, {$R$}-charges, and an efficient inverse
  algorithm.
\newblock {\em J. High Energy Phys.}, (10):014, 31, 2008.

\bibitem{Hanany-Herzog-Vegh_BTEC}
Amihay Hanany, Christopher~P. Herzog, and David Vegh.
\newblock Brane tilings and exceptional collections.
\newblock {\em J. High Energy Phys.}, (7):001, 44 pp. (electronic), 2006.

\bibitem{Hanany-Kennaway_DMTD}
Amihay Hanany and Kristian~D. Kennaway.
\newblock Dimer models and toric diagrams.
\newblock hep-th/0503149, 2005.

\bibitem{Hanany-Vegh}
Amihay Hanany and David Vegh.
\newblock Quivers, tilings, branes and rhombi.
\newblock {\em J. High Energy Phys.}, (10):029, 35, 2007.

\bibitem{Ishii-Ueda_DMSMCv1}
Akira Ishii and Kazushi Ueda.
\newblock Dimer models and the special {McKay} correspondence.
\newblock arXiv:0905.0059v1.

\bibitem{Ishii-Ueda_08}
Akira Ishii and Kazushi Ueda.
\newblock On moduli spaces of quiver representations associated with dimer
  models.
\newblock In {\em Higher dimensional algebraic varieties and vector bundles},
  RIMS K\^oky\^uroku Bessatsu, B9, pages 127--141. Res. Inst. Math. Sci.
  (RIMS), Kyoto, 2008.

\bibitem{Kenyon-Schlenker}
Richard Kenyon and Jean-Marc Schlenker.
\newblock Rhombic embeddings of planar quad-graphs.
\newblock {\em Trans. Amer. Math. Soc.}, 357(9):3443--3458 (electronic), 2005.

\bibitem{King}
A.~D. King.
\newblock Moduli of representations of finite-dimensional algebras.
\newblock {\em Quart. J. Math. Oxford Ser. (2)}, 45(180):515--530, 1994.

\bibitem{Mercat_DRSIM}
Christian Mercat.
\newblock Discrete {R}iemann surfaces and the {I}sing model.
\newblock {\em Comm. Math. Phys.}, 218(1):177--216, 2001.

\bibitem{Mozgovoy-Reineke}
Sergey Mozgovoy and Markus Reineke.
\newblock On the noncommutative {D}onaldson-{T}homas invariants arising from
  brane tilings.
\newblock arXiv:0809.0117.

\bibitem{Van_den_Bergh_NCR}
Michel van~den Bergh.
\newblock Non-commutative crepant resolutions.
\newblock In {\em The legacy of Niels Henrik Abel}, pages 749--770. Springer,
  Berlin, 2004.

\end{thebibliography}

\ \\

\noindent
Akira Ishii

Department of Mathematics,
Graduate School of Science,
Hiroshima University,
1-3-1 Kagamiyama,
Higashi-Hiroshima,
739-8526,
Japan

{\em e-mail address}\ : \ akira@math.sci.hiroshima-u.ac.jp

\ \\

\noindent
Kazushi Ueda

Department of Mathematics,
Graduate School of Science,
Osaka University,
Machikaneyama 1-1,
Toyonaka,
Osaka,
560-0043,
Japan.

{\em e-mail address}\ : \  kazushi@math.sci.osaka-u.ac.jp

\end{document}